\documentclass[a4paper,12pt]{article}

\usepackage[pdftex]{graphicx}
\usepackage{amsmath,amsthm,amsfonts,amssymb,amscd, amsxtra,color}
\usepackage[active]{srcltx}
\usepackage{ctable}
\usepackage{verbatim}
\usepackage{url}
\usepackage{cite}
\usepackage{enumerate}
\usepackage[margin=2cm,nohead]{geometry}
\usepackage{multicol}
\usepackage{graphics}
\usepackage{bbold}
\usepackage{threeparttable}
\usepackage{indentfirst}

\theoremstyle{plain}

\newtheorem{definition}{Definition}[section]

\newtheorem{lemma}{Lemma}[section]
\newtheorem{corollary}{Corollary}[section]
\newtheorem{proposition}{Proposition}[section]

\newtheorem{example}{Example}[section]

\newtheorem{theorem}{Theorem}[section]

\DeclareMathOperator{\LCP}{LCP}
\DeclareMathOperator{\StLCP}{SLCP}
\DeclareMathOperator{\SStLCP}{SOL-SLCP}
\DeclareMathOperator{\StMCP}{SMixCP}
\DeclareMathOperator{\SStMCP}{SOL-SMixCP}

\DeclareMathOperator{\C}{\mathcal C}

\newcommand{\lf}{\left}
\newcommand{\rg}{\right}
\newcommand{\R}{\mathbb R}

\newcommand{\E}{\mathbb E}

\newcommand{\tp}{^\top}

\newcommand{\bs}{\left(\begin{smallmatrix}}
\newcommand{\es}{\end{smallmatrix}\right)}

\begin{document}

\title{Stochastic Linear Complementarity Problems on Extended Second Order Cones
}
\author{S. Z. N\'emeth\\School of Mathematics, The University of Birmingham\\The Watson Building, Edgbaston\\Birmingham B15 2TT, United Kingdom\\email: s.nemeth@bham.ac.uk
\and L. Xiao\\School of Mathematics, The University of Birmingham\\The Watson Building, Edgbaston\\Birmingham B15 2TT, United Kingdom\\email: Lxx490@bham.ac.uk}
\maketitle
\paragraph{}

\begin{abstract}
In this paper, we study the stochastic linear complementarity problems on
extended second order cones (stochastic ESOCLCP). We first convert the
problem to a stochastic mixed complementarity problem on the nonegative
orthant (SMixCP). Enlightened by the idea of Chen and Lin(2011), we
introduce the Conditional Value-at-risk (CVaR) method to measure the loss
of complementarity in the stochastic case. A CVaR - based minimisation
problem is introduced to achieve a solution which is "good enough" for
the complementarity requirement of the original SMixCP. Smoothing function
and sample average approximation methods are introduced and the the
problem is converted to a form which can be solved by Levenberg-Marquardt
smoothing SAA algorithm. At the end of the paper a numerical example
illustrates our results.

\end{abstract}

\section{Introduction}

Uncertainty is a common and realistic problem that results from inaccurate measurement or stochastic variation of data such as price, capacities, loads, etc. In fact, the inaccuracy or uncertainty of these real-world data are inevitable. When these data are applied as parameters in mathematical models, the constraints of models may be violated because of their stochastic characters. These violations may finally cause some difficulties in the sense that the optimal solutions obtained from the stochastic data are no longer optimal, or even infeasible. Amongst approaches proposed for modeling uncertain quantities, the stochastic models outstand because of their solid mathematical foundations, theoretical richness, and sound techniques of using real data. Complementarity problems imbedded with stochastic models occur in many areas such as finance, telecommunication and engineering. Hence, considering $\LCP$ with uncertainty will be meaningful for practical treatments. If partial or all of the coefficients in the $\LCP$ are uncertain, the $\LCP$ will be turned into a stochastic linear complementarity problem ($\StLCP$), which is firstly introduced by Chen and Fukushima \cite{chen2005expected}. Articles about $\StLCP$ can be found in \cite{fang2007stochastic,gurkan1999sample,lin2009combined,chen2011cvar}.

Even though the fact that only limited number of results have been obtained on the stochastic complementarity problems, there are still some meaningful results. One of them is the CVaR (conditional value-at-risk, which is also called expected shortfall) minimisation reformulation of stochastic complementarity problem \cite{xu2014cvar}. In this study, the stochastic linear complementarity problem on extended second order cones (S-ESOCLCP) will be introduced. Based on the results in \cite{nemeth2017linear}, a method of finding solutions to S-ESOCLCP will be elaborated, and a numerical example will be presented.
~\\

\section{Preliminaries}
~\\

For $\R^k$, the Euclidian space whose elements are column vectors, the definition of an inner product $\langle \cdot, \cdot \rangle$ is given by
\[
    \langle x, y\rangle =\sum\limits_{i=1}^n x_i y_i, \quad x, y \in \R^n.
\]

Let $k$, $l$ be positive integers. The inner product of pairs of vectors $\bs x\\u\es \in \R^k\times\R^{\ell}$, where $x\in\R^k$ and $u\in\R^{\ell}$, is given by
\[
    \langle\bs x\\u\es ,\bs y\\v\es \rangle =\langle x,y\rangle + \langle u,v\rangle .
\]

Let $\R^n$ be a Euclidian space.  A set $K \subseteq\R^n$ is called a \emph{convex cone} if for any $\alpha , \beta > 0$, and $x, y\in K$, we have
\[
    \alpha x + \beta y \in K.
\]
In other words, a convex cone is a set which is invariant with respect to the multiplication of vectors with positive scalars. The {\it dual cone} of a convex cone ${K} \subseteq\R^n$ is
\[
    {K}^*\!\!:=\!\{ x\in \R^n : \langle x, y \rangle\!\geq\! 0, ~ \forall \, y\!\in\! {K}\}.
\]
which is also a convex cone. A convex cone \( {K} \subseteq\R^{n}\)  is called {\it pointed} if \({ K} \cap \{-{K}\} \subseteq \{0\}\), or equivalently, if \({K}\) does not contain straight  lines through the origin. A convex cone which is a closed set is called a \emph{closed convex cone}. Any pointed closed convex cone with nonempty interior will be called {\it proper cone}. The cone ${K}$ is called {\it subdual} if ${K}\subseteq {K}^*$, {\it superdual} if ${K}^*\subseteq {K}$, and {\it self-dual} if ${K}^*={K}$.

\begin{definition}
    An indicator function $\mathbf{1}_{I}\lf(\theta\rg)$
    is defined as
    \[
        \bold{1}_{I}\lf(\theta\rg)=
        \left\{
    		\begin{array}{l}
                1 \qquad if\; \theta\in I \\
                0 \qquad otherwise.
            \end{array}
    		\right.
    \]
\end{definition}

\begin{definition}[Complementarity Set]\label{def:cs}
     Let $K\subseteq\R^m$ be a nonempty closed convex cone and $K^*$ its dual\cite{MR0321540}. The set
    \[
        \C(K):=\left\{(x,y)\in K\times K^*:\langle x, y\rangle = 0\right\}
    \]
    is called the \emph{complementarity set} of $K$.
\end{definition}
Denote by $\|\cdot\|$ the corresponding Euclidean norm. Recall the definitions of the mutually dual extended second order cone \cite{NZ20151} $L(k,\ell)$ and $M(k,\ell)$ in $\mathbb{R}^n\equiv\R^k\times\R^\ell$:
\begin{equation}\label{elc}
L(k,\ell) = \{(x,u) \in \mathbb{R}^k\times \mathbb{R}^\ell : x \geq \|u\|e\},
\end{equation}
\begin{equation}\label{delc}
M(k,\ell) = \{(x,u) \in \mathbb{R}^k\times \mathbb{R}^\ell : e\tp x\geq \|u\|,x\ge0\},
\end{equation}
where  $e=(1, \dots, 1)\tp \in \mathbb{R}^k $. If there is no ambiguity about the dimensions, then we simply denote $L(k,\ell)$ and $M(k,\ell)$ by
$L$ and $M$, respectively. The following proposition is from \cite{nemeth2017linear}. We restate it here for convenience.

\begin{proposition}\label{prop:cs-esoc}
	Let $x,y\in\R^k$ and $u,v\in\R^\ell\setminus\{ 0\}$.
	\begin{enumerate}
		\item[(i)] $(x,0,y,0)\in\C(L)$ if and only if $(x,y)\in\C(\R^k_+)$.
		\item[(ii)] $(x,0,y,v)\in\C(L)$ if and only if $e\tp y\ge\|v\|$ and $(x,y)\in\C(\R^k_+)$.
		\item[(iii)] $(x,u,y,0)\in\C(L)$ if and only if $x\ge\|u\|e$ and $(x,y)\in\C(\R^k_+)$.
        \item[(iv)] $(x,u,y,v):=(\bs x\\u\es,\bs y\\v\es)\in\C(L)$ if and only if there exists $\lambda>0$ such that $v=-\lambda u$,
			$e\tp y=\|v\|$ and $(x-\|u\|e,y)\in C(\R^k_+)$.
	\end{enumerate}
\end{proposition}

The linear complementarity problem $\LCP(F,L)$ defined by a convex cone $K$ and a linear function $F(x,u)$ is given by
\begin{equation}\label{que:lcp}
        \LCP(F,L)\left\{
        \begin{array}{l}
			Find \; \bs x\\u\es,\; such \; that\\
			\lf( \bs x\\u\es, F(x,u) \rg) \in \C(K)
		\end{array}
        \right.
\end{equation}
where $F(x,u) = T\bs x\\u\es +r$, $T=\bs A & B\\C & D \es$, $r = \bs p \\ q\es$.

\section{The problem formulation}

Let $(\Omega, \mathcal{F},\mathcal{P})$ be a probability space defined by:

\begin{enumerate}
  \item $\Omega\subseteq\R^{n}$, the sample set of possible outcomes;
  \item $\mathcal{F}\subseteq2^{\Omega}$, a $\sigma$-algebra generated by $\Omega$ (or a collection of all subsets of $\Omega$); and
  \item $\mathcal{P}: \mathcal{F}\rightarrow [0,1]$, a function maps from events to probabilities.
\end{enumerate}

The stochastic linear complementarity problem $SLCP(F,L,\omega)$ defined by a convex cone $L$ and a linear function $F(x,u,\omega)$ is given by
\begin{equation}\label{queslcp}
    SLCP(F, K,\omega)\left\{
    \begin{array}{l}
    	Find \; \bs x\\u\es,\;  such \; that\\
    	 \lf( \bs x\\u\es, F(x,u,\omega) \rg)\in \C(K) , \omega\in \Omega, \quad a.s.
    \end{array}
    \right.
\end{equation}
where $F(x,u,\omega) = T(\omega)\bs x\\u\es+r(\omega): \R^m \times \Omega \rightarrow \R^m$, $\omega\in\Omega$ is an $n$-dimensional random vector. The notation ``a.s." is the abbreviation of ``alomst surely", which means
\[
    \mathcal{P}\{\omega \in \Omega; x\in \R^m_+, F(x,\omega) \ge 0, x\tp F(x,\omega)=0\} = 1
\]

In this study, we assume that the coefficients $T(\omega)$ and $r(\omega)$ are measurable functions of $\omega$ with the following property:
\[
    \E[\|T(\omega)\tp T(\omega)\|] < \infty\quad and \quad \E[\|r(\omega)\|] < \infty
\]
where $\E[\cdot]$ represents the expected value of the random vector in the square bracket.

It should be noted that if the possible outcome set $\Omega$ contains only one single realisation (and this unique outcome definitely happens), problem (\ref{queslcp}) will degenerate to (\ref{que:lcp}).

The stochastic linear complementarity problems are very useful in solving practical problems. However, because of the existence of the random vector $\omega$ in the function $F (x, \omega)$, it is very difficult and sometimes impossible of finding a solution $x$ to satisfy almost all possible outcomes of $\omega$. One plausible idea to improve the viability of finding a solution to SLCP is to associate the problems with probability models, and then persuasive solutions to SLCP are obtainable by finding the solutions to the associated probability models.

Xu and Yu \cite{xu2014cvar} summarised 6 different probability models for finding solutions to SLCP:


\begin{enumerate}
    \item[(i)] \textbf{Expected value (EV) method, introduced by G{\"u}rkan et. al in \cite{gurkan1999sample}}. By using the expectation value $\E[F(x,\omega)]$ to replace the stochastic term $F(x,\omega)$, this method ultimately reformulates (\ref{queslcp}) to (\ref{que:lcp}).

    \item[(ii)] \textbf{Expected residual minimisation (ERM) method, introduced by Chen and Fukushima \cite{chen2005expected}}. This method minimises the expectation of the square norm of the residual $\Phi(x,\omega)$ defined by the following C-function:
        \begin{equation}\label{CF_max}
            \min_{x\in\R^n_+}\E \left[ \|\Phi(x,\omega)\|^2\right]
        \end{equation}
        where $\Phi:\R^m \times\Omega\rightarrow\R^m$ is a multi dimensional C-function defined as
        \[
            \Phi(x,\omega):= \lf(\phi\lf(x_1,F_1(x,\omega)\rg),\dots,\phi\lf(x_m,F_m(x,\omega)\rg)\rg)\tp .
        \]
        where $\phi : \R\times\R \rightarrow\R$ can be any scalar C-function satisfying:
        \[
            \phi(a,b) = 0 \quad \Leftrightarrow \quad a\ge 0,\quad b\ge 0, \quad ab = 0.
        \]

    \item[(iii)] \textbf{Stochastic mathematical programs with equilibrium constraints (SMPEC) reformulation, introduced by Lin and Fukushima\cite{lin2009combined, lin2009solving, mataramvura2008risk}}. This method highlights a recourse variate $z(\omega)$ to compensate the violation of complementarity in (\ref{queslcp}) for some outcomes of $\omega\in\Omega$, then it reformulates (\ref{queslcp}) to the following model:
        \begin{equation}\label{que:SMPEC}
         \begin{array}{lll}
        & \min\limits_x & \E\left[\eta\lf(z(\omega)\rg)\right] \\
        & s.t.    & 0\le x\perp \lf(F(x,\omega) + z(\omega)\rg) \ge 0, \\
        &          & z(\omega) \ge 0, \omega\in \Omega\quad a.s.,
         \end{array}
        \end{equation}
        where $\eta(z) = e^{tp}z$. Ambiguous solutions to SCP can be obtained by minimising the objective function in (\ref{que:SMPEC}), i.e. the expected value of the compensation to the violation of complementarity in (\ref{queslcp}).

    \item[(iv)] \textbf{Stochastic programming (SP) reformation \cite{wang2010stochastic}}. Problem (\ref{queslcp}) is reformulated to the following:
        \[
        \begin{array}{lll}
            & \min\limits_{x} & \E\left[\|\lf(x \circ F(x,\omega)\rg)_+\|^2 \right]  \\
            & s.t.           & F(x,\omega) \ge 0, \quad \omega \in \Omega \quad a.s. \\
            &                 & x\ge 0.
        \end{array}
        \]
        where $x_+:=\max\{x,0\}$, and $x\circ F(x,\omega)$ is the Hadamard product of $x$ and $F(x,\omega)$.

    \item[(v)] \textbf{Robust Optimisation \cite{ben2002robust, ben2009robust}, which is a deterministic reformation of  (\ref{queslcp})}. And,

    \item[(vi)] \textbf{CVaR minimisation (CM) reformulation \cite{chen2011cvar}}.  By using this method, (\ref{queslcp}) is reformulated to a problem that minimises the CVaR  of the norm of the loss function $\theta(x,\omega)$, namely:
        \[
            \min_{x\in\R^n} CVaR_{\alpha}\lf(\| \theta(x,\omega)\|\rg).
        \]

\end{enumerate}

The reformulation in item (vi) uses the CVaR, a measure of risk widely applied in financial industry. CVaR was built based on Value at risk (VaR)\cite{rockafellar2000optimization,meucci2009risk}. Let $ \omega\in\Omega$ be a vector with random outcomes and let $\theta(x,\omega): \R^m\times\Omega \rightarrow \R$ be a mapping, the VaR of $\omega$ for the loss function is defined as:

\begin{equation}\label{VaR}
    VaR_\alpha(\theta(x,\omega)) =  min\{\Theta | \mathcal{P}[\theta(x,\omega)\ge \Theta] \le \alpha\}.
\end{equation}
where $\mathcal{P}[\cdot]\in[0,1]$ is the probability of the event in the square bracket. We call $\theta(x,\omega)$ the loss function. The probability (also called confidence level) $\alpha \in (0,1)$ quantifies the proportion of ``worst cases" (that is, $\theta(x,\omega)\ge VaR_\alpha(\theta(x,\omega))=\Theta$) in the group of all outcomes, and the other outcomes ($\theta(x,\omega)< \Theta$) would happens with probability $1-\alpha$.

Based on the definition of VaR, CVaR is defined as:

\begin{align}
    CVaR_\alpha(\theta(x,\omega)) & = \frac{1}{\alpha}\E\lf[\theta(x,\omega)\bold{1}_{[ VaR_{\alpha}\lf(\theta(x,\omega)\rg),+\infty)} \lf(\theta(x,\omega)\rg)\rg] \label{que:var2cvarexp}\\
        & = \frac{1}{\alpha}\int_{\theta(x,\omega)\ge VaR_{\alpha}(\theta(x,\omega))} \theta(x,\omega) d\mathcal{P}(\omega) \nonumber\\
        & = \frac{1}{\alpha}\int_{0}^{\alpha}VaR_{\gamma}(\theta(x,\omega))d\gamma. \label{que:var2cvar}
\end{align}

$CVaR_\alpha(\theta(x,\omega))$ is the conditional expectation of all outcomes with $\theta(x,\omega)\ge VaR_\alpha (\theta(x,\omega))$. For better understanding the concept of VaR and CVaR, figure \ref{fig:cvar} gives a sample of a loss function $\theta(x,\omega) = \omega$ with one-dimensional normally distributed random value $\omega \sim N(0,1)$. This figure shows that when the confidence level $(1-\alpha)$ is set at $95\%$, the value of VaR equals to the horizontal coordinate of the red vertical line, and the value of CVaR with $95\%$ confidence level equals the red area to the right of the line.

\begin{figure}[h!]
    \centering
    \begin{minipage}{1\textwidth}
    \includegraphics[width=1\textwidth, height=0.5\textwidth]{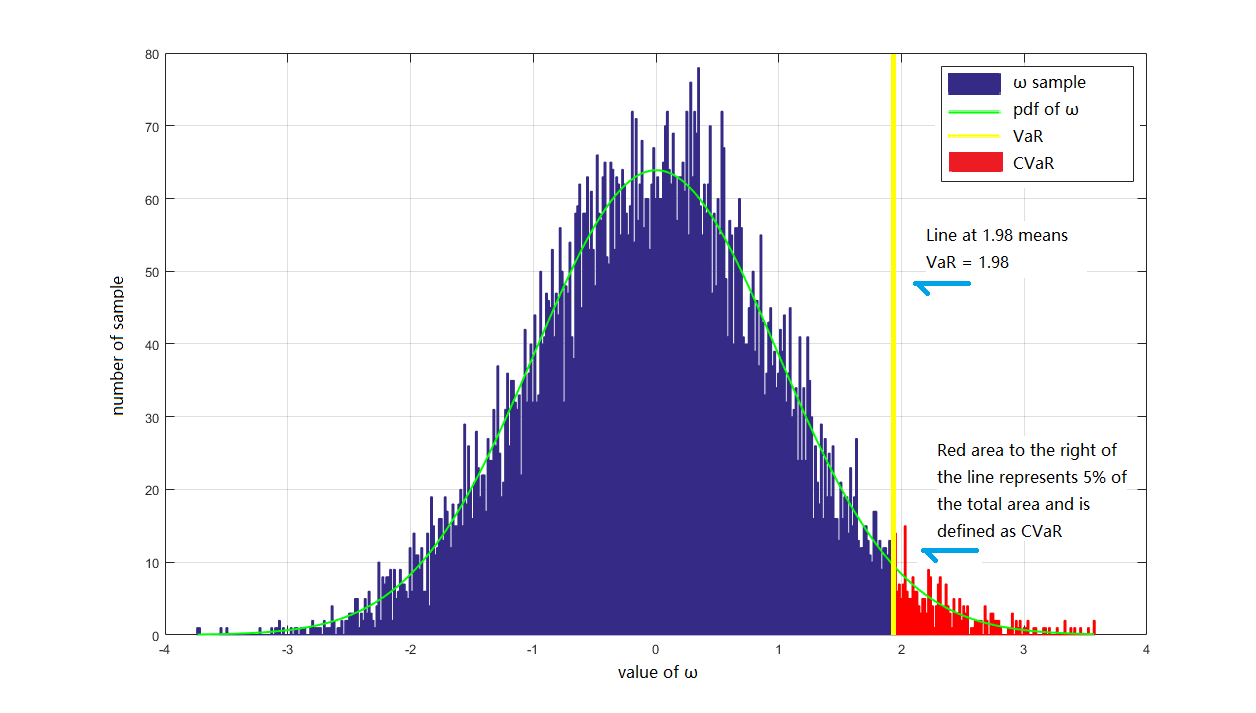}\\
    {\footnotesize For a stochastic event $x$ with $\E(x) = 0$, the distribution of this event shows that only $5\%$ of the outcomes are above 1.98. If the confidence level is set at $95\%$, then the value of VaR equals 1.98(horizontal axis marked by yellow line), and the value of CVaR equals the integral of the area marked in red color. \par}
    \end{minipage}
    \caption{VaR and CVaR for $\theta(x,\omega) = \omega$, where $\omega \sim N(0,1)$}
    \label{fig:cvar}
\end{figure}

\begin{proposition}
    A risk measure $S\lf(\theta(x,\omega)\rg)$ has the following properties:
    \begin{enumerate}
        \item Positive homogeneity: $S\lf(\lambda\theta(x,\omega)\rg) = \lambda S\lf(\theta(x,\omega)\rg)$ for all $\lambda >0$ and $\omega\in \Omega$,
        \item Monotonicity: if $\theta(x_1,\omega) \ge \theta(x_2,\omega)$ for all $\omega \in \Omega$, we have $S\lf(\theta(x_1,\omega)\rg) \ge S\lf(\theta(x_2,\omega)\rg)$, and
        \item Sub-additivity: $S\lf(\theta(x_1,\omega)+ \theta(x_2,\omega)\rg) \le S\lf(\theta(x_1,\omega)\rg) + S\lf(\theta(x_2,\omega)\rg)$  for all $\omega \in \Omega$.
    \end{enumerate}
\end{proposition}

\begin{proposition}\label{Var_prop}
    The risk measure VaR is
    \begin{enumerate}
        \item Positive homogeneous, and
        \item Monotonic.
    \end{enumerate}
\end{proposition}

\begin{proof}

    \begin{enumerate}

    \item By the definition of VaR, we have:
    \begin{align*}
        VaR_\alpha(\lambda\theta(x,\omega)) & = min\{\Theta' | \mathcal{P}[\lambda\theta(x,\omega)\ge \Theta'] \le \alpha\} \\
         & = min\{\lambda\Theta | \mathcal{P}[\lambda\theta(x,\omega)\ge \lambda\Theta] \le \alpha\} \\
         & = \lambda min\{\Theta | \mathcal{P}[\theta(x,\omega)\ge \Theta] \le \alpha\} \\
         & = \lambda VaR_\alpha(\theta(x,\omega))
    \end{align*}

    \item Suppose $\theta(x_1,\omega) \ge \theta(x_2,\omega)$ for all $\omega \in \Omega$. We have
    \[
        \mathcal{P}[\theta(x_1,\omega)\ge \Theta] \ge \mathcal{P}[\theta(x_2,\omega)\ge \Theta],
    \]
    therefore
    \[
        VaR_\alpha(\theta(x_1,\omega)) \ge VaR_\alpha(\theta(x_2, \omega)).
    \]

    \end{enumerate}
\end{proof}
We mark that VaR is not sub-additive (counter example is shown in \cite{danielsson2005subadditivity}). 
\begin{example}
    Consider the function $\theta(x,\omega) = x + \omega$, where
    \[
        \omega = \epsilon + \eta, \;\; \epsilon \sim N(0,1),\;\; \eta = \lf\{
        \begin{array}{ll}
    		0 & with\; probability\; 0.991\\
    		10 & with\;probability\; 0.009
    	\end{array}
        \rg.
    \]
    In the case when $\alpha = 0.01$, we obtain
    \[
        VaR_{\alpha}\lf(\theta(x,\omega) +\theta(y,\omega)\rg) = x+y +9.8 > VaR_{\alpha}\lf(\theta(x,\omega)\rg) + VaR_{\alpha}\lf(\theta(y,\omega)\rg) = x+3.1 + y + 3.1.
    \]
\end{example}

\begin{proposition}
    The risk measure CVaR is
    \begin{enumerate}
        \item Positive homogeneous,
        \item Monotonic, and
        \item sub-additive\cite{acerbi2002coherence}.
    \end{enumerate}
\end{proposition}

\begin{proof}
    The positive homogeneity and monotonicity of CVaR can be easily proved by combining the Proposition \ref{Var_prop} and the definition \eqref{que:var2cvar}. To prove the Sub-additivity define $\theta_3 = \theta_1 + \theta_2$. By definition \eqref{que:var2cvarexp}, we have
    \footnotesize
    \begin{align*}
        \alpha  \cdot(CVaR_\alpha(\theta_1) & + CVaR_\alpha(\theta_1) - CVaR_\alpha(\theta_3))  = \\
        & = \alpha\cdot \E\lf[
              \frac{1}{\alpha}\theta_1\bold{1}_{[VaR_{\alpha}(\theta_1),+\infty)} (\theta_1)
            + \frac{1}{\alpha}\theta_2\bold{1}_{[VaR_{\alpha}(\theta_2),+\infty)} (\theta_2)
            - \frac{1}{\alpha}\theta_3\bold{1}_{[VaR_{\alpha}(\theta_3),+\infty)} (\theta_3)
            \rg] \\
        & = \E\lf[
              \theta_1\lf(\bold{1}_{[VaR_{\alpha}(\theta_3),+\infty)} (\theta_3) - \bold{1}_{[VaR_{\alpha}(\theta_1),+\infty)} (\theta_1)\rg)\rg. \\
            & \qquad\qquad\qquad\qquad\qquad\qquad \qquad\qquad +
              \lf.\theta_2\lf(\bold{1}_{[VaR_{\alpha}(\theta_3),+\infty)} (\theta_3) - \bold{1}_{[VaR_{\alpha}(\theta_2),+\infty)} (\theta_2)\rg)
            \rg] \\
        & \ge VaR_{\alpha}(\theta_1)\E\lf[\bold{1}_{[VaR_{\alpha}(\theta_3),+\infty)} (\theta_3) - \bold{1}_{[VaR_{\alpha}(\theta_1),+\infty)} (\theta_1)\rg]\\
            & \qquad\qquad\qquad\qquad\qquad\qquad \qquad\qquad +
            VaR_{\alpha}(\theta_2)\E\lf[\bold{1}_{[VaR_{\alpha}(\theta_3),+\infty)} (\theta_3) - \bold{1}_{[VaR_{\alpha}(\theta_2),+\infty)} (\theta_2)\rg] \\
        & = VaR_{\alpha}(\theta_1)\lf(\alpha - \alpha\rg) + VaR_{\alpha}(\theta_2)\lf(\alpha - \alpha\rg) \\
        & = 0 \\
    \end{align*}
\end{proof}

Consider $\StLCP(F, L,\omega)$ defined by the function $F(x,u,\omega) = T(\omega)\bs x\\u\es  + r(\omega)$ and the extended second order cone $L$, problem (\ref{queslcp}) becomes:

\[
    \StLCP(F, L,\omega)\lf\{
    \begin{array}{l}
		Find \; (x,u)\in L,\; such \; that\\
		F(x,u,\omega) \in M \; and \; \langle \bs x\\u\es, F(x,u,\omega) \rangle = 0,\;\omega\in \Omega, \quad a.s.
	\end{array}
    \rg.
\]
where $T(\omega)=\bs A(\omega) & B(\omega)\\C(\omega) & D(\omega) \es$, with $A(\omega)\in\R^{k\times k}$,
	$B(\omega)\in\R^{k\times\ell}$, $C(\omega)\in\R^{\ell\times k}$ and $D(\omega)\in\R^{\ell\times\ell}$; $r(\omega) = \bs p(\omega)\\ q(\omega)\es$, with $p(\omega)\in \R^{k}$, $q(\omega)\in \R^{\ell}$, for $\omega\in\Omega$.

We use Corollary \ref{cor:Main_thm} to reformulate the SLCP to the stochastic mixed complementarity problem (SMixCP). Given the functions $\widetilde{F}_1$, $\widetilde{F}_2$, and a cone $K$, the SMixCP is defined as:
	\[
		\StMCP(\widetilde{F}_1,\widetilde{F}_2, K ,\omega):\left\{
		\begin{array}{l}
			Find \; \bs x\\u\\t\es, \; such \; that\\
			\widetilde{F}_2(x,u,t,\omega)=0, \; and\; (x,\widetilde{F}_1(x,u,t,\omega))\in\C(K),\;\omega\in \Omega, \quad a.s.
		\end{array}
		\right.
	\]

\begin{corollary}\label{cor:Main_thm}

     Suppose $u\ne 0$, $Cx+Du+q\ne 0$. We have
            \[
                z\in \SStLCP(F, L,\omega)\iff \exists t>0,
            \]
            such that
            \[
                \tilde{z}\in \SStMCP(\widetilde{F}_1,\widetilde{F}_2,\R^k_+,\omega),
            \]
            where
            \[
                \widetilde{F}_1(x,u,t,\omega)=A(\omega)(x+te)+B(\omega)u+p(\omega)
            \]
            and
            \small
            \begin{equation}\label{mathcal_SF}
                \widetilde{F}_2(x,u,t,\omega) =
                \begin{pmatrix}
                    \lf[tC(\omega)+ue\tp A(\omega)\rg](x+te)+ue\tp\lf[B(\omega)u+p(\omega)\rg]+t\lf[D(\omega)u+q(\omega)\rg] \\
                    t^2 - \|u\|^2
                \end{pmatrix}.
            \end{equation}
            \normalsize
\end{corollary}

\begin{proof}
    Suppose that $z\in\SStLCP(F,L,\omega)$. Then, $(x,u,y,v)\in\C(L)$, where $y=A(\omega)x+B(\omega)u+p(\omega)$ and $v=C(\omega)x+D(\omega)u+q(\omega)$. Let $t = \|u\|$, Then, by item (iv)
	of Proposition \ref{prop:cs-esoc} we have that $\exists\lambda>0$ such that
	\begin{equation}\label{parall-eq3}
		C(\omega)x+D(\omega)u+q(\omega)=v=-\lambda u,
	\end{equation}
	\begin{equation}\label{prod-eq3}
		e\tp (A(\omega)x+B(\omega)u+p(\omega))=e\tp y=\|v\|=\|C(\omega)x+D(\omega)u+q(\omega)\|=\lambda t,
	\end{equation}
	\begin{equation}\label{cpset-eq3}
		(\tilde{x},\widetilde{F}_1(\tilde{x},u,t,\omega))=(x-te,A(\omega)x+B(\omega)u+p)=(x-te,y)\in\C(\R^k_+)
	\end{equation}
	where $\widetilde{z} = \bs \tilde{x}\\u\\t\es := \bs x-t\\u\\t\es\in\R^{k+\ell+1}$. From equation \eqref{parall-eq3} we obtain $t(C(\omega)x+D(\omega)u+q(\omega))=-t \lambda u$, which by equation \eqref{prod-eq3} implies
	$t(C(\omega)x+D(\omega)u+q(\omega))=-ue\tp (A(\omega)x+B(\omega)u+p(\omega))$, which after some algebra gives
	\begin{equation}\label{zero-eq3}
		\widetilde{F}_2(\tilde{x},u,t)= 0.
	\end{equation}
	From equations \eqref{cpset-eq3} and \eqref{zero-eq3} we obtain that $z\in\SStMCP(\widetilde{F}_1,\widetilde{F}_2,\R^k_+,\omega)$.
	\medskip
\end{proof}
Corollary \ref{cor:Main_thm} provides an alternative way to find the solutions to the $\StLCP(F,L,\omega)$, by converting it to the $\StMCP(\widetilde{F}_1,\widetilde{F}_2, \R^k_+,\omega)$. Such conversion enables us to study $\StLCP(F,L,\omega)$ through a C-function.
The scalar form of \emph{Fischer-Burmeister (FB) C-function} \cite{fischer1992special, fischer1995newton} is defined as:

\[
    \psi_{FB}(a,b) = \sqrt{a^2+b^2} - (a+b) \quad \forall (a,b) \in \mathbb{R}^2.
\]

The FB-based equation formulation of $\StMCP(\widetilde{F}_1,\widetilde{F}_2,\R^k_+,\omega)$ is:

\begin{equation}\label{que:FB-SMixCP}
    \mathbb{F}^{\StMCP}_{FB}(x,u,t,\omega) =
    \begin{pmatrix}
         \psi\lf(x_1,(\widetilde{F}_1)_1(x,u,t,\omega)\rg) \\
         \vdots \\
         \psi\lf(x_k,(\widetilde{F}_1)_k(x,u,t,\omega)\rg) \\
         \widetilde{F}_2(x,u,t,\omega)
    \end{pmatrix}.
\end{equation}
where $\psi(\cdot ):\R^2\rightarrow \R$ is the scalar FB C-function stated in Chapter 2. It should be mentioned that the FB C-function is convex, but non-smooth on $\psi(0,0)$. According to the definition of FB C-function, a point $(x^*,u^*,t^*)$ is a solution to the stochastic mixed complementarity problem $\StMCP(\widetilde{F}_1,\widetilde{F}_2,\R^k_+,\omega)$, if and only if
\begin{equation}\label{que:stfb}
    \mathbb{F}^{\StMCP}_{FB}(x,u,t,\omega) = 0.
\end{equation}

The associated merit function of $\StMCP(\widetilde{F}_1,\widetilde{F}_2,\R^k_+,\omega)$ is:

\begin{equation}\label{que:MF-SMixCP}
    \theta^{\StMCP}_{FB}(x,u,t,\omega)=\frac{1}{2}\mathbb{F}^{\StMCP}_{FB}(x,u,t,\omega)\tp \mathbb{F}^{\StMCP}_{FB}(x,u,t,\omega).
\end{equation}
Based on (\ref{que:FB-SMixCP}) and (\ref{que:MF-SMixCP}), the merit function can be written as:
\[
    \theta_{FB}^{\StMCP}(x,u,t,\omega) = \frac{1}{2}\sum\limits_{i=1}^{k}\psi^2\lf(x_i,\widetilde{F}_1^i(x,u,t,\omega)\rg) + \frac{1}{2}\sum\limits_{j=1}^{\ell}\widetilde{F}_2^j(x,u,t,\omega).
\]
By the definition of merit function, a point $(x^*,u^*,t^*)$ is a solution to the stochastic mixed complementarity problem $\StMCP(\widetilde{F}_1,\widetilde{F}_2,\R^k_+,\omega)$, if

\[
	\theta_{FB}^{\StMCP}(x^*,u^*,t^*,\omega) = 0 \quad \omega\in \Omega, \quad a.s.
\]

\begin{proposition}
    The associated merit function $\theta_{FB}^{\StMCP}(x^*,u^*,t^*,\omega)$ is continuously differentiable on $\R^{k+\ell+1}$, if $\widetilde{F}_1(x^*,u^*,t^*,\omega)$ and $\widetilde{F}_2(x^*,u^*,t^*,\omega)$ are continuously differentiable on $\R^k$ and $\R^{\ell+1}$, respectively.
\end{proposition}
\begin{proof}
    First we prove that $\psi^2$ is continuously differentiable. We note that $\psi$ is continuously differentiable at every $(a, b)\neq (0,0)$. It is easy to verify that $\psi^2$ is continuously differentiable at every $(a, b)\neq (0,0)$. Consider the following to limits at point $(a, b) = (0,0)$:
    \[
        \lim\limits_{\Delta x\rightarrow 0}\frac{\psi^2(\Delta x, 0) - \psi^2(0, 0)}{\Delta x} = \frac{2\lf(\Delta x^2\rg) - 2\sqrt{\Delta x^2} \cdot \Delta x}{\Delta x} = 2(\Delta x - |\Delta x|) = 0,
    \]
    and
    \[
        \lim\limits_{\Delta y\rightarrow 0}\frac{\psi^2(0, \Delta y) - \psi^2(0, 0)}{\Delta y} = \frac{2\lf(\Delta y^2\rg) - 2\sqrt{\Delta y^2} \cdot \Delta y}{\Delta y} = 2(\Delta y - |\Delta y|) = 0.
    \]
where $\Delta x$, $\Delta y > 0$. Both partial derivatives of $\psi^2$ at $(0,0)$ are continuous, $\psi^2$ is continuously differentiable. Hence, $\theta_{FB}^{\StMCP}(x^*,u^*,t^*,\omega)$ is continuously differentiable on $\R^{k+\ell+1}$ if and only if $\widetilde{F}_1(x^*,u^*,t^*,\omega)$ and $\widetilde{F}_2(x^*,u^*,t^*,\omega)$ are continuously differentiable on $\R^k$ and $\R^{\ell+1}$, respectively.
\end{proof}

Since the function $\psi^2(a,b)$ is not convex on $\R^2$, it is easy to verify that the merit function $ \theta_{FB}^{\StMCP}(x,u,t,\omega)$ is not convex on the feasible region. We focus on finding the stationary point of the merit function \eqref{que:FB-SMixCP}. A point $(x^*,u^*,t^*)$ is a stationary point to the stochastic mixed complementarity problem $\StMCP(\widetilde{F}_1,\widetilde{F}_2,\R^k_+,\omega)$, if:

\begin{equation}\label{que:SMIXCP}
    \nabla \theta_{FB}^{\StMCP}(x^*,u^*,t^*,\omega) =  0 \quad \omega\in \Omega, \quad a.s.
\end{equation}

That is
\begin{equation}\label{que:SMIXCP2}
    \mathcal{A}(\omega)^\top\mathbb{F}_{FB}^{\StMCP}(x^*,u^*,t^*,\omega) =  0 \quad \omega\in \Omega, \quad a.s.,
\end{equation}
where
\[
    \mathcal{A}=
    \begin{pmatrix}
        D_a+D_bJ_x\widetilde{F}_1(x^*,u^*,t^*,\omega) & D_bJ_{(u,t)}\widetilde{F}_1(x^*,u^*,t^*,\omega)\\
        J_x\widetilde{F}_2(x^*,u^*,t^*,\omega) & J_{(u,t)}\widetilde{F}_2(x^*,u^*,t^*,\omega)
    \end{pmatrix}
\]
is a nonsingular matrix. Combining equation \eqref{que:SMIXCP2} with equation \eqref{que:stfb} implies that equation \eqref{que:SMIXCP} is a necessary condition for $(x^*,u^*,t^*)$ to be a solution to $\StMCP(\widetilde{F}_1,\widetilde{F}_2,\R^k_+,\omega)$.

The feasible set of $\StMCP(\widetilde{F}_1,\widetilde{F}_2,\R^k_+,\omega)$ shrinks as $|\Omega|$ (i.e., the size of the possible outcome set $\Omega$) increases. When $|\Omega| = \infty$, we cannot generally find a solution to $\StMCP(\widetilde{F}_1,\widetilde{F}_2,\R^k_+,\omega)$ such that system (\ref{que:SMIXCP}) holds almost surely for any $\omega\in \Omega$, because there will be a large number of equations in system \eqref{que:SMIXCP}. Figure \ref{fig:Omega} shows the situation when the size of $\Omega$. 

As it is introduced above, probability models provide appropriate deterministic reformulations of the stochastic complementarity problems. It can be associated with the stochastic complementarity problems to find persuasive solutions. These persuasive solutions to stochastic complementarity problems would make a proper trade-off between the satisfaction of infinite complementarity constraints and solvability of the problems.

\begin{figure}[h!]
  \centering
  \begin{minipage}{1\textwidth}
  \includegraphics[width=1\textwidth, height=1\textwidth]{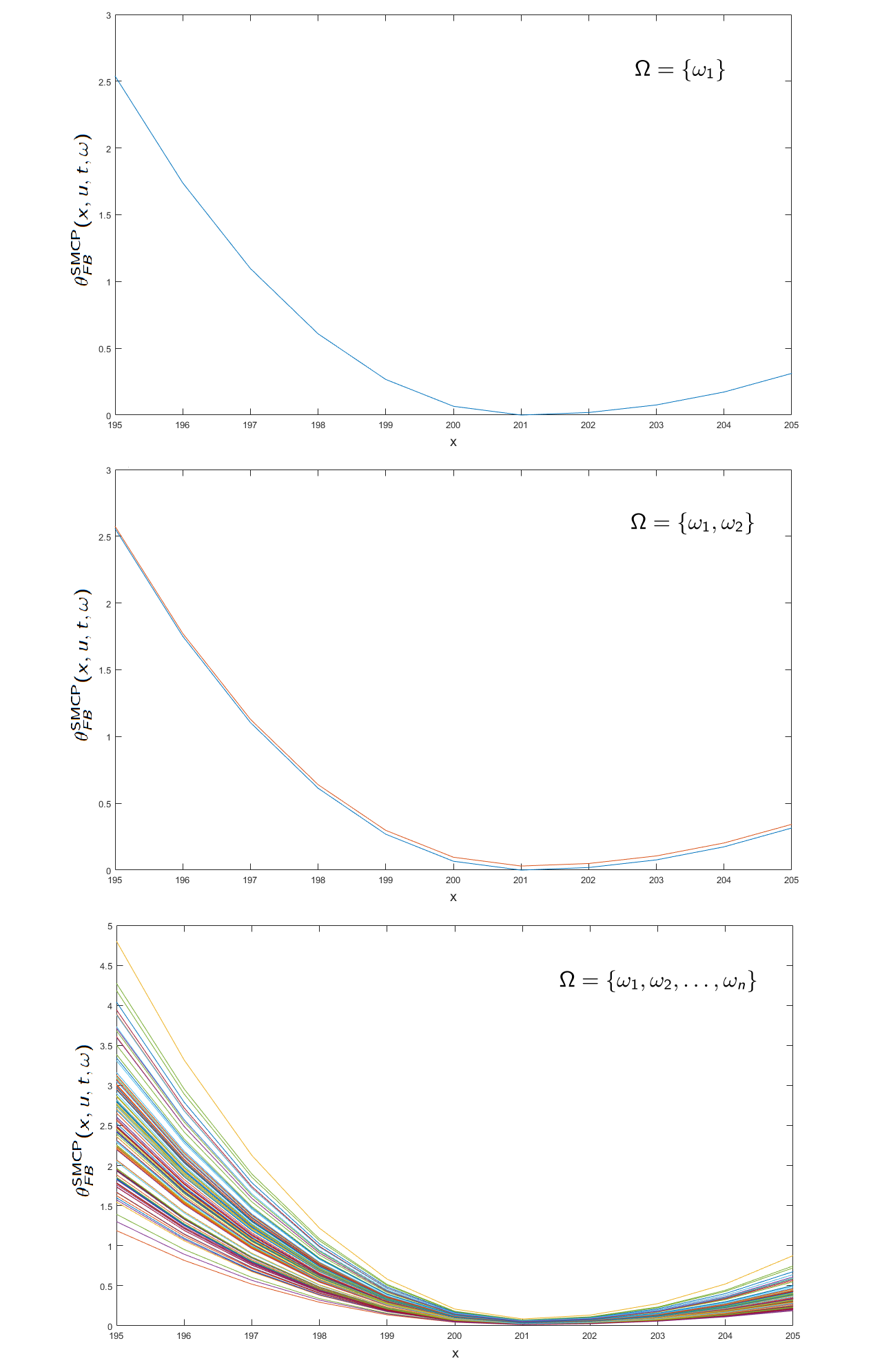}\\
  {\footnotesize For a possible outcome set $\Omega$, when the size of $\Omega$ equals 1, i.e. $|\Omega|=1$ (figure 1), we can easily find a solution to the problem by using the merit function. When $|\Omega|$ increases to 2 (figure 2), the solution for the first case is not longer suitable for both outcomes. As the size of $|\Omega|$ increases (figure 3), it become almost impossible to find a solution to the problem which is suitable for all outcomes. \par}
    \end{minipage}
  \caption{The minimum point of merit function varies $\theta^{\StMCP}_{FB}$ as $|\Omega|$ increases}
  \label{fig:Omega}
\end{figure}

Since $\theta^{\StMCP}_{FB}(x, u, t,\omega)\ge 0$, given a confidence level $(1-\alpha) \in (0,1)$, a point $(x^*, u^*, t^*)$ is a plausible solution to $\StMCP(\widetilde{F}_1,\widetilde{F}_2,\R^k_+,\omega)$ if it is a solution to

\begin{equation}\label{que:slacksmcp}
    \arg \min_{x,u,t}\{\Theta|\mathcal{P}\{\theta^{\StMCP}_{FB}(x^*, u^*, t^*,\omega) \le \Theta\} \ge 1 - \alpha\},
\end{equation}

This is a relaxation of problem (\ref{que:SMIXCP}). A small value of $\alpha$ means that it prefers the satisfaction of the complementarity constraints rather than the solvability of the problem. Note that the problem (\ref{que:slacksmcp}) can be written as:

\begin{equation}\label{Ind_fun}
    \arg \min_{x,u,t}\{\Theta|\E[\bold{1}_{[0,+\infty)}\lf(\theta^{\StMCP}_{FB}(x^*, u^*, t^*,\omega) - \Theta\rg)]\le \alpha\},
\end{equation}
where the indicator function $\bold{1}_{[0,+\infty)}(\cdot)$ is neither convex nor continuously differentiable at the point $0$. Even though the function $\theta^{\StMCP}_{FB}(\cdot)$ is convex and continuously differentiable, the objective function (\ref{Ind_fun}) is non-smooth. If we use the indicator function in the objective function, difficulties occur when applying algorithms which are only viable for smooth objective functions. Addressing this concern, the CVaR method will be considered, which undertakes convex and continuously differentiable objective functions. It harmonises the incompatibility between the satisfaction of infinite number of complementarity constraints and solvability of the problems, as well as inherits convexity and continuous differentiability from the merit function $ \theta^{\StMCP}_{FB} (x,u,t,\omega)$. In the CVaR method, $ (\theta^{\StMCP}_{FB}(x^*, u^*, t^*,\omega) - \Theta)$ will be used as the ``loss function" to measure the ``loss" of complementarity. It should be noted that, the higher the value of the ``loss function", the more complementarity constraints of this stochastic complementarity problem are lost. We transform (\ref{que:slacksmcp}) into CVaR based objective function and then construct the stochastic programming model  in the following context.

Rewritting (\ref{Ind_fun}) as Value-at-Risk (VaR) to measure of the loss of complementarity:
\[
    VaR_{\alpha}\lf(\theta^{\StMCP}_{FB}(x^*, u^*, t^*,\omega)-\Theta\rg)\le 0
\]

VaR is a measure of complementarity loss defined in (\ref{VaR}). However, the disadvantages of using VaR as the measure of complementarity loss is significant: VaR is not consistent, because it is neither convex nor smooth \cite{artzner1999coherent}. On the other hand, CVaR (defined in (\ref{que:var2cvar})) has superior mathematical properties than VaR, as it inherits continuous differentiability and convexity from the merit function. Moreover, CVaR is a more conservative measure of complementarity loss than VaR.

\begin{theorem}\label{thm:cvar_cont}
    If $\theta^{\StMCP}_{FB}(x,u,t,\omega)$ is continuously differentiable on $I\subseteq\R^{k+\ell+1}$, then for any $0<\alpha<1$, the measure of complementarity loss $CVaR_\alpha(\theta^{\StMCP}_{FB}(x,u,t,\omega))$ is continuously differentiable on $\R_+$.
\end{theorem}

\begin{proof}
    Immediate from the continuous differentiability of $\theta^{\StMCP}_{FB}(x,u,t,\omega)$ and (\ref{que:var2cvar}).
\end{proof}

\begin{theorem}\label{thm:cvar_conv}
    If $\theta^{\StMCP}_{FB}(x,u,t,\omega)$ is convex on $I\subseteq\R^{k+\ell+1}$, then for any $0<\alpha<1$, the measure of complementarity loss $CVaR_\alpha(\theta^{\StMCP}_{FB}(x,u,t,\omega))$ is convex on $I\subseteq\R^{k+\ell+1}$.
\end{theorem}

\begin{proof}
    Suppose that $\theta^{\StMCP}_{FB}(z,\omega)$ is convex on $I$. Let $z$, $z'\in I\subseteq\R^{k+\ell+1}$. Then we have
    \[
        \theta^{\StMCP}_{FB}(\lambda z+(1-\lambda)z',\omega) \le \lambda\theta^{\StMCP}_{FB}(z,\omega) + (1-\lambda)\theta^{\StMCP}_{FB}(z',\omega),
    \]
    where $\lambda\in[0,1]$. By the sub-additivity and homogeneity of CVaR, we have
    \begin{align*}
        CVaR_\alpha &(\theta(\lambda z+(1-\lambda)z',\omega)) \\
         & = \frac{1}{\alpha}\int_{0}^{\alpha}VaR_{\gamma}\lf(\theta(\lambda z+(1-\lambda)z',\omega)\rg)d\gamma \\
         & \le \frac{1}{\alpha}\int_{0}^{\alpha}VaR_{\gamma}\lf(\lambda\theta^{\StMCP}_{FB}(z,\omega) + (1-\lambda)\theta^{\StMCP}_{FB}(z',\omega)\rg)d\gamma \\
         & \le \frac{1}{\alpha}\int_{0}^{\alpha}\lf[VaR_{\gamma}\lf(\lambda\theta^{\StMCP}_{FB}(z,\omega)\rg) + VaR_{\gamma}\lf((1-\lambda)\theta^{\StMCP}_{FB}(z',\omega)\rg)\rg]d\gamma \\
         & = \frac{\lambda}{\alpha}\int_{0}^{\alpha}VaR_{\gamma}\lf(\theta^{\StMCP}_{FB}(z,\omega)\rg)d\gamma + \frac{1-\lambda}{\alpha}\int_{0}^{\alpha} VaR_{\gamma}\lf(\theta^{\StMCP}_{FB}(z',\omega)\rg)d\gamma \\
         & = \lambda CVaR_\alpha (\theta(z,\omega))  + (1-\lambda)CVaR_\alpha (\theta(z',\omega)).
    \end{align*}

    Hence, $ CVaR_\alpha(\theta^{\StMCP}_{FB}(x,u,t,\omega))$ is convex in $z$.

\end{proof}

\begin{definition}
    Suppose $S_1\lf(\theta(x,\omega)\rg)$, $S_2\lf(\theta(x,\omega)\rg): \R^n \rightarrow I$ are two risk measures. Risk measure $S_1\lf(\theta(x,\omega)\rg)$ is said to be \emph{more conservative} than risk measure $S_2\lf(\theta(x,\omega)\rg)$ if
    \[
        S_1\lf(\theta(x,\omega)\rg) \ge S_2\lf(\theta(x,\omega)\rg)
    \]
    for all $x\in \R^n$ and $\omega\in \Omega$.
\end{definition}

\begin{proposition}
    For the problem $\StMCP(\widetilde{F}_1,\widetilde{F}_2,\R^k_+,\omega)$, $CVaR_\alpha(x)$ is a more conservative measure of complementarity loss than $VaR_{\alpha}(x)$.
\end{proposition}

\begin{proof}
    By definition (\ref{que:var2cvar}) we have:
    \begin{align*}
        CVaR_\alpha\lf(\theta^{\StMCP}_{FB}(x,u,t,\omega)\rg) & = \frac{1}{\alpha}\int^{\alpha}_{0} VaR_\tau \lf(\theta^{\StMCP}_{FB}(x,u,t,\omega)\rg)d\tau \\
                    & = \E[VaR_\tau\lf(\theta^{\StMCP}_{FB}(x,u,t,\omega)\rg)|0\le\tau\le \alpha] \\
                    & \ge \min\{VaR_\tau\lf(\theta^{\StMCP}_{FB}(x,u,t,\omega)\rg)|0\le\tau\le \alpha\} \\
                    & = VaR_\alpha\lf(\theta^{\StMCP}_{FB}(x,u,t,\omega)\rg)
    \end{align*}

    Hence, we have
    \[
        CVaR_\alpha(\theta^{\StMCP}_{FB}(x,u,t,\omega))  \ge VaR_\alpha (\theta^{\StMCP}_{FB}(x,u,t,\omega)).
    \]
\end{proof}

Reformulate the problem $\StMCP (\widetilde{F}_1, \widetilde{F}_2, \R^k_+,\omega)$ to the following CVaR based minimisation problem:
\begin{equation}\label{que:mincvar}
    \min_{(x,u,t)\in\R^{k+\ell+1}}CVaR_{\alpha}(\theta^{\StMCP}_{FB}(x,u,t,\omega)),
\end{equation}
where
\[
    CVaR_\alpha(\theta^{\StMCP}_{FB}(x,u,t,\omega)) = \alpha^{-1}\int_{0}^{\alpha}VaR_\gamma\lf(\theta^{\StMCP}_{FB}(x,u,t,\omega)\rg)d\gamma,
 \]
and
\[
    VaR_\alpha(\theta^{\StMCP}_{FB}) =  min\{\Theta | \mathcal{P}[\theta^{\StMCP}_{FB}(x,u,t,\omega)\ge \Theta] \le \alpha\}.
\]

It means that a solution $(x^*,u^*,t^*)$ to $\StMCP$ should minimise the ``loss" of complementarity from stochasticity.

Let
\[
    [t]_+ := \max\{0,t\},
\]
\[
    \nu_{(\Theta,\alpha)} (x,u,t, \omega) := \Theta + \alpha^{-1}[\theta^{\StMCP}_{FB}(x,u,t, \omega) - \Theta]_+,
\]
and define
\[
    \mathcal{N}_{\alpha}(x,u,t, \omega,\Theta):= \E\lf[\nu_{(\Theta,\alpha)} (x,u,t, \omega)\rg] =  \Theta + \alpha^{-1}\E[\theta^{\StMCP}_{FB}(x,u,t, \omega) - \Theta]_+.
\]

\begin{lemma}\label{prop:cvareqv}
    The problem (\ref{que:mincvar}) is equivalent to the following problem:
    \begin{equation}\label{que:cvareqv}
        \begin{array}{lcl}
            & \min\limits_{(x,u,t)\in\R^{k+\ell+1}} &\mathcal{N}_{\alpha}(x,u,t, \omega,\Theta^*) \\
        \end{array}
    \end{equation}
    where $\Theta^*$ is the optimal value satisfying:
    \[
        \Theta^* \in \arg\inf\limits_{\Theta\in\R} \lf\{ \mathcal{N}_{\alpha}(x,u,t, \omega,\Theta)\rg\}.
    \]
\end{lemma}
\begin{proof}
    Immediate from the alternative definition of CVaR \cite{rockafellar2002conditional}:
    \[
        CVaR_{\alpha} (\theta^{\StMCP}_{FB}(x,u,t,\omega)) := \inf\limits_{\Theta\in\R}\lf\{\Theta + \alpha^{-1}\E[\theta^{\StMCP}_{FB}(x,u,t, \omega) - \Theta]_+\rg\}.]
    \]
\end{proof}

Problem (\ref{que:cvareqv}) is simplified from (\ref{que:mincvar}) because it does not contain integration, and inherits convexity from the merit function $\theta^{\StMCP}_{FB}(x,u,t, \omega)$. However, since the presence of the operator $[\cdot]_+$, the objective function in problem (\ref{que:cvareqv}) is not smooth at the point 0. Using mathematical techniques to smooth the objective function will be more convenient and applicable than directly using semi-smooth algorithm to find the numerical solution. Four palmary smoothing functions summarised by Chen and Harker \cite{chen1997smooth} are provided as follows:
\begin{enumerate}
    \item[(i)] Neural network smoothing function:
    \[
        p(t,\mu) = t + \mu\log(1 + e^{-\frac{t}{\mu}}).
    \]

    \item[(ii)] Interior point smoothing function:
    \[
        p(t,\mu) = \frac{t+\sqrt{t^2 + 4\mu}}{2}.
    \]

    \item[(iii)] Auto-scaling interior point smoothing function:
    \[
        p(t,\mu) = \frac{t+\sqrt{t^2 + 4\mu^2}}{2} + \mu.
    \]

    \item[(iv)] Chen-Harker-Kanzow-Smale (CHKS) smoothing function:
    \[
        p(t,\mu) = \frac{t+\sqrt{t^2 + 4\mu^2}}{2}.
    \]
\end{enumerate}
where $\mu\ge 0$ is the parameter of the approximation function $p$. It should be noted that:
\[
    \lim_{\mu\rightarrow +0}p(t,\mu) = [t]_+.
\]

In this paper, we choose Chen-Harker-Kanzow-Smale (CHKS) smoothing function and denote:
\[
    [t]_\mu =  \frac{t+\sqrt{t^2 + 4\mu^2}}{2}.
\]

We rewrite problem (\ref{que:cvareqv}) as:
\[
    \begin{array}{lcc}
       & \min\limits_{(x,u,t)\in\R^{k+\ell+1},\Theta\in\R} & \mathcal{N}_{\alpha}(x,u,t, \omega,\Theta) = \Theta + \alpha^{-1}\E[\theta^{\StMCP}_{FB}(x,u,t, \omega) - \Theta]_\mu
   \end{array}
\]

The mathematical expectation is another difficulty that needs to be carefully treated. In many instances, the mathematical expectation $\E[\cdot]$ cannot be calculated with accuracy. A common treatment is using the Sample Average Approximation (SAA) method, which is based on the Law of large numbers. SAA method provides a persuasive result of measuring an expectation value \cite{gurkan1999sample, jiang2008stochastic}. If the distribution of the random vector  $\omega$ is known, then the Monte-Carlo approach can be used to generate a sample independently and identically distributed (i.i.d.) $\{\omega^1, \dots, \omega^N\}$ with the distribution of $\omega$. Let $\{\omega^1, \dots, \omega^N\}$ be an i.i.d. sample set. The SAA method estimates the mathematical expectation $\E[\theta^{\StMCP}_{FB}(x,u,t, \omega) - \Theta]_\mu$ using averaged value of all observations $[\theta^{\StMCP}_{FB}(x,u,t, \omega^1) - \Theta]_\mu,\, [\theta^{\StMCP}_{FB}(x,u,t, \omega^2) - \Theta]_\mu,\, \dots,\,[\theta^{\StMCP}_{FB}(x,u,t, \omega^N) - \Theta]_\mu$. That is,

\begin{align*}
    \hat{\mathcal{N}_{\alpha}}^N(x,u,t,\Theta)  := & \;\frac{1}{N}\sum\limits^N_{i=1}\mathcal{N}_{\alpha}(x,u,t, \omega^i,\Theta)\\
                                                                = & \;\Theta + \alpha^{-1}\frac{1}{N}\sum\limits^N_{i=1}[\theta^{\StMCP}_{FB}(x,u,t, \omega^i) - \Theta]_\mu.
\end{align*}

Then, problem (\ref{que:cvareqv}) becomes
\begin{equation}\label{que:cvarsaa}
    \begin{array}{lcc}
       & \min\limits_{(x,u,t)\in\R^{k+\ell+1},\Theta\in\R} & \hat{\mathcal{N}_{\alpha}}(x,u,t, \Theta) = \Theta + \alpha^{-1}\frac{1}{N}\sum\limits^N_{i=1}[\theta^{\StMCP}_{FB}(x,u,t, \omega^i) - \Theta]_\mu \\
       & s.t.                                                 & (x,u,t) \in \R^{k}\times\R^{\ell}\times\R, \quad \Theta\in \R.
   \end{array}
\end{equation}

The gradient of $ \hat{\mathcal{N}_{\alpha}}(x,u,t, \Theta) $ is:

\[
    \nabla\hat{\mathcal{N}_{\alpha}}(x,u,t, \Theta) =
    \begin{pmatrix}
        \nabla_{x,u,t}\mathcal{N}_{\alpha}(x,u,t, \Theta) \\
        \nabla_{\Theta}\mathcal{N}_{\alpha}(x,u,t, \Theta)
    \end{pmatrix}
\]
where
\small
\begin{equation}\label{que:nu1}
     \nabla_{x,u,t}\hat{\mathcal{N}_{\alpha}}(x,u,t, \Theta)   =  \alpha^{-1}\frac{1}{2N}\sum\limits^N_{j=1} \lf[1+\frac{\theta^{\StMCP}_{FB}(x,u,t, \omega^j) - \Theta }{\sqrt{\lf(\theta^{\StMCP}_{FB}(x,u,t, \omega^j) - \Theta\rg)^2+4\mu}}\rg]\mathcal{A}^\top_j \mathbb{F}_{FB}^{\StMCP}(x,u,t, \omega^j),
\end{equation}
\normalsize

\begin{equation}\label{que:nu2}
    \mathcal{A}_j =
    \begin{pmatrix}
        D_{a,j}+D_{b,j}\widetilde{A}_{j} & D_{b,j}\widetilde{B}_{j}\\
        \widetilde{C}_{j} & \widetilde{D}_{j}
    \end{pmatrix},
\end{equation}
\[
    D_{a,j}= diag
    \begin{pmatrix}
        \frac{x_i}{\sqrt{(x_i)^2 + \widetilde{F}_1^i(x,u,t, \omega^j)^2}} - 1
    \end{pmatrix}
    \qquad
    D_{b,j}=diag
    \begin{pmatrix}
        \frac{\widetilde{F}_1^i(x,u,t, \omega^j)}{\sqrt{(x_i)^2 + \widetilde{F}_1^i(x,u,t, \omega^j)^2}} - 1
    \end{pmatrix},
    \qquad
    i=1, \dots , k,
\]

\[
    \widetilde{A}_{j} = A(\omega^j), \qquad
    \widetilde{B}_{j} =
    \begin{pmatrix}
    B(\omega^j) & A(\omega^j)e
    \end{pmatrix}, \qquad
        \widetilde{C}_{j} =
    \begin{pmatrix}
        tC(\omega^j)+ue^\top A(\omega^j) \\
        \mathbf 0
    \end{pmatrix},
\]

\[
    \widetilde{D}_{j} =\scriptsize
    \begin{pmatrix}
        \lf[A(\omega^j)(x+te)+B(\omega^j)u+p(\omega^j)\rg]\tp eI + ue\tp B(\omega^j) + tD(\omega^j) & C(\omega^j)x+2tC(\omega^j)e+ue^\top A(\omega^j)e+D(\omega^j)u \\
        -2u^\top & 2t
    \end{pmatrix}
    \normalsize,
\]

and
\begin{equation}\label{que:nu3}
    \nabla_{\Theta}\hat{\mathcal{N}_{\alpha}}(x,u,t, \Theta)  = 1 - \alpha^{-1}\frac{1}{N}\sum\limits^N_{j=1} \lf[\frac{1}{2}+\frac{\theta^{\StMCP}_{FB}(x,u,t, \omega^j) - \Theta }{2\sqrt{\lf(\theta^{\StMCP}_{FB}(x,u,t, \omega^j) - \Theta\rg)^2+4\mu}}\rg].
\end{equation}

Since the objective function $ \hat{\mathcal{N}_{\alpha}}(x,u,t, \Theta) $ is continuously differentiable, Problem \eqref{que:cvarsaa} can be solved by finding some solutions $(x^*,u^*,t^*, \Theta^*)$ to
\begin{equation}\label{que:grad}
    \nabla\hat{\mathcal{N}_{\alpha}}(x,u,t, \Theta)\ = 0.
\end{equation}

\section{An algorithm}

In the previous section, we have modified the $\StLCP(F, L,\omega)$ to the problem \eqref{que:cvarsaa} with a convex and continuously differentiable objective function. Furthermore, the solution to the $\StLCP(F, L,\omega)$ can be obtained by finding some solutions $(x^*,u^*,t^*, \Theta^*)$ to equation \eqref{que:grad}. In this section, an algorithm will be developed to solve \eqref{que:grad}. This algorithm contains Monte-Carlo approach to generate i.i.d. random vector sample sets. We denote $z = (x,u,t) \in \R^{m+1}$. Given the tolerance $r >0$, the stopping criterion is:

\begin{equation}\label{que:saatol}
    \max_i\lf\{\lf\|\frac{\partial\mathcal{N}_{\alpha}^{(N_j,\mu_t)}(x,u,t, \omega,\Theta)}{\partial z}\rg\|\rg\} \le r.
\end{equation}
for $i = 1,\dots, m+1$. It is shown as follows:

\begin{flushleft}
\textbf{Algorithm 3} (Line search smoothing SAA):
\end{flushleft}

\textbf{Input}: initial point $z_0 = (x_0,u_0,t_0) \in \R^{k}\times\R^{\ell}\times\R $, $\Theta_0\in \R$, confidence level $\alpha$, LM parameter $\nu$, the smoothing parameter $\mu$, maximum iteration number $j_{max}$ for $j$, $k_{max}$ for $k$, the sequence of sample set sizes $ N_1<N_2<\dots<N_{j_{max}}$, parameters of the approximation $\nu$, $\mu$, the tolerance $r>0$, $\varepsilon>0$, and parameters for Wolfe conditions $c_1$, $c_2\in (0,1)$.

\textbf{Step 1}: Set  $j = 1$.

\textbf{Step 2}: Set the sample size $N = N_j$, and generate i.i.d samples $\{\omega^1, \dots, \omega^{N}\}$.

\textbf{Step 3}: If $j>1$, and $\lf\| z^j - z^{j-1}\rg\| < \varepsilon$, \textbf{Stop}.

\textbf{Step 4}: Set  $k = 0$, and $y_0 = z_0$.

\textbf{Step 5}: If either (\ref{que:saatol}) or $k = k_{max}$, then set $j = j + 1$, $z_j = y_k$, and go to \textbf{Step 3}.

\textbf{Step 6}: Denote $\bar{\mathcal{A}_j} = \frac{1}{N}\sum^{N_j}_{i=1}\mathcal{A}_i$, and find a direction $d_k \in \R^{k}\times\R^{\ell}\times\R$ such that
\begin{equation}\label{que:salgo}
     \bar{\mathcal{A}_j}(y_k)\tp \mathbb{F}^{\StMCP}_{FB}(y_k) + \left[\bar{\mathcal{A}_j}\tp (y_k)\bar{\mathcal{A}_j}(y_k) +\mu \mathbb{I}\right]d_k = 0.
\end{equation}

If the system \eqref{que:salgo} is not solvable or if the condition
\[
    \nabla\hat{\mathcal{N}_{\alpha}}(y_k, \Theta)\tp d_k \le -r\|d_k\|
\]
is not satisfied, (re)set $d_k = -\nabla\hat{\mathcal{N}_{\alpha}}(y_k, \Theta)$.

\textbf{Step 7}: Find step length $s_k \in R_+$ such that
\[
    \hat{\mathcal{N}_{\alpha}}(y_k + s_k d_k, \Theta) \le \hat{\mathcal{N}_{\alpha}}(y_k + s_k d_k, \Theta) + c_1s_k\nabla\hat{\mathcal{N}_{\alpha}}(y_k, \Theta)\tp d_k,
\]
and
\[
    \hat{\mathcal{N}_{\alpha}}(y_k + s_k d_k, \Theta)\tp d_k \ge c_2 \nabla \hat{\mathcal{N}_{\alpha}}(y_k, \Theta)\tp d_k.
\]

\textbf{Step 8}: Set $y_{k+1} := y_k + s_kd_k $ and $k := k + 1$, go to \textbf{Step 5}.

\textbf{Comment}: This algorithm requires the Monte-Carlo approach to generate i.i.d. random vector samples. If the values of $N_j$, $j = 1,\dots, j_{max}$ are large, the algorithm is anticipated to be more accurate, but it will sacrifice time and computing power. On the other hand, if the value of $N_j$'s are small, the costs of finding result is decreased, but the accuracy of the solution is sacrificed.

\section{A numerical example}
This section illustrates a numerical example for the stochastic ESOCLCP. Denote by $L(3,2)$ an extended second order cone in $\R^{(3+2)}$. Let $x\in\R^3$ and $u\in\R^2$ be two real vectors. Denote
\[
    z=\bs x\\u\es\in\R^{3+2},\quad \hat z=\bs x-\|u\|\\u\es\in\R^{3+2},\quad and \; \tilde{z}=\bs x-t\\u\\t\es\in\R^{3+2+1}.
\]

A stochastic ESOCLCP defined by the extended second order cone $L(3,2)$ and a stochastic linear function $F(x,u,\omega) = T(\omega)\bs x\\u\es + r(\omega)$ is:
\[
    SLCP(T(\omega),r(\omega),L(3,2))\left\{
    \begin{array}{l}
    	Find \; x\in L(3,2),\; such \; that\\
    	T(\omega)x+r(\omega) \ge 0, x\tp (T(\omega)x+r(\omega)) = 0,\; \omega\in \Omega, \quad a.s. ,
    \end{array}
    \right.
\]
where
\[
    T=\begin{pmatrix} A & B\\C & D \end{pmatrix} =
    \lf(
    \begin{array}{crcrr}
      41+\omega_1 &  -3 &   -31 &  18 & 19  \\
      28 &  22 & -33 &  25 &  -29  \\
     -23 & -29 &  11 & -21 &  -43  \\
     - 9 & -31 & -20+2\omega_2 & -12 & 47  \\
     - 8 &  46 &  50 &  -22 & 21
    \end{array}
    \rg),
\]
with $A\in\R^{3\times 3}$, $B\in\R^{3\times 2}$, $C\in\R^{2\times 3}$ and $D\in\R^{2\times 2}$; and
\[
    r=\lf(\begin{array}{r}p\\ q\end{array}\rg) = \lf(
    \begin{array}{c}
     -26\\
       4 - \omega_3\\
      23\\
      44\\
     -19
    \end{array}\rg),
\]
with $p\in\R^3$, $q\in\R^2$. $\omega = \lf(\omega_1, \omega_2, \omega_3\rg)\tp \in\Omega$ is a stochastic vector with i.i.d. random variables $\omega_i \sim N(0,1)$ for all $i = 1,2, 3$. Again it is easy to verify that square matrices T, A and D are nonsingular for all $\omega_i \in \R$, $i = 1,2,3$.

By using Corollary \ref{cor:Main_thm}, we reformulate $SLCP(T(\omega),r(\omega),L(3,2))$ to the following $\StMCP$ defined by $\widetilde{F}_1$, $\widetilde{F}_2$, and $\R^3_+$:
\[
	\StMCP(\widetilde{F}_1,\widetilde{F}_2, \R^3_+,\omega):\left\{
	\begin{array}{l}
		Find \; \bs x\\u\\t\es\in\R^{3+2+1}, \; such \; that\\
		\widetilde{F}_2(x,u,t,\omega)=0, \; and\; (x,\widetilde{F}_1(x,u,t,\omega))\in\C(\R^3_+),\;\omega\in \Omega, \quad a.s.
	\end{array}
	\right.
\]
where
            \[
                \widetilde{F}_1(x,u,t,\omega)=A(\omega)(x+te)+B(\omega)u+p(\omega)
            \]
            and
            \small
            \[
                \widetilde{F}_2(x,u,t,\omega) =
                \begin{pmatrix}
                    & \lf(tC(\omega)+ue\tp A(\omega)\rg)(x+te)+ue\tp(B(\omega)u+p(\omega))+t(D(\omega)u+q(\omega)) \\
                    & t^2 - \|u\|^2
                \end{pmatrix}.
            \]
            \normalsize
We will convert this to the form of \eqref{que:cvarsaa} and then \eqref{que:grad}. Given $\alpha = 0.05$, we rewrite problem \eqref{que:cvarsaa} as:
\[
    \min\limits_{(x,u,t)\in\R^{3+2+1},\Theta\in\R} \Theta + 0.05^{-1}\frac{1}{N}\sum\limits^{N}_{i=1}[\theta^{\StMCP}_{FB}(x,u,t, \omega^i) - \Theta]_\mu,
\]
where
\[
    \theta_{FB}^{\StMCP}(x,u,t,\omega) = \frac{1}{2}\sum\limits_{i=1}^{3}\psi^2\lf(x_i,\widetilde{F}_1^i(x,u,t,\omega)\rg) + \frac{1}{2}\sum\limits_{j=1}^{2}\widetilde{F}_2^j(x,u,t,\omega).
\]

Since the distribution of the random vector $\omega$ is known, we use the Monte Carlo (MC) method to simulate $j_{max}$ sample sets with number of observation $N_1, N_2, \dots, N_{j_{max}}$. The solutions are shown in the following table:

\begin{table}[!ht]
\centering
\begin{threeparttable}
\begin{tabular}{c r c c c}
\toprule
    $j$ & $N_j$ & $(x,u)\tp$ & $F(x,u,\omega)\tp$ \\[0.4ex]
\midrule
     1 &      10 & (1.537, 0.273, 1.060, 0.136, -0.262) & (0.784, 29.054, -0.194, -13.466, 25.803) \\
     2 &     100 & (1.542, 0.263, 1.058, 0.127, -0.253) & (1.093, 28.552, -0.214, -12.609, 25.544) \\
     3 &    1000 & (1.549, 0.257, 1.060, 0.122, -0.252) & (1.277, 28.397, -0.162, -12.418, 25.477) \\
     4 &   10000 & (1.548, 0.262, 1.060, 0.125, -0.254) & (1.215, 28.605, -0.204, -12.701, 25.578) \\
     5 &  100000 & (1.546, 0.261, 1.059, 0.125, -0.254) & (1.186, 28.587, -0.176, -12.643, 25.516) \\
     6 & 1000000 & (1.546, 0.261, 1.059, 0.124, -0.254) & (1.200, 28.566, -0.177, -12.617, 25.514) \\
\bottomrule
\end{tabular}
\end{threeparttable}
\end{table}

\begin{table}[!ht]
\centering
\begin{threeparttable}
\begin{tabular}{c r c c c}
\toprule
    $j$ & $N_j$ & Run time(s) & Average loss of complementarity & Value of threshold $\Theta$ \\[0.4ex]
\midrule
    1 &       10 & 0.090439 & 0.347 & 0.063  \\
    2 &      100 & 0.696431 & 0.893 & 0.095  \\
    3 &     1000 & 5.202383 & 1.179 & 0.090  \\
    4 &    10000 & 39.39705 & 1.060 & 0.087  \\
    5 &   100000 & 553.4596 & 1.054 & 0.088  \\
    6 &  1000000 & 4759.294 & 1.073 & 0.089  \\
\bottomrule
\end{tabular}
\begin{tablenotes}
\footnotesize
\item[] Note: The first table shows the solution of $SLCP(T(\omega),r(\omega),L(3,2))$ and the value of the function $F(x,u,\omega)\tp$ with respect to different value of $N$. The value of solution does not variate significantly, while the value of the function differs but converges to around 1.200 as the value of $N$ increase. The second table shows the Run time(s), average loss of complementarity, and the value of threshold. The run time increases significantly as the value $N$ increases. On the other hand, the average loss of complementarity and the value of threshold remains relative constant no matter what change to the value of $N$.\\
\end{tablenotes}
\end{threeparttable}
\caption{The result of numerical example}
\end{table}

\bigskip

The average loss of complementarity (ALoC) is calculated by:
\[
    ALoC = \frac{1}{N_j}\sum_{i=1}^{N_j}\|(x,u)\tp F(x,u,\omega_i)\|.
\]

 As it is shown in the table above, by the cost of increased run time, the dispersion of solution reduced as the number of iteration $N_j$ increases. However, increasing the number of iteration $N_j$ does not decrease neither the value of the threshold $\Theta$ nor the ALoC.

\section{Conclusion and comments}

In this paper, we studied the stochastic linear complementarity problems on extended second order cones (stochastic ESOCLCP) that are stochastic extensions of ESOCLCP included in our previous paper \cite{nemeth2018linear}. We used Corollary \ref{cor:Main_thm}, which states the equivalent conversion from a stochastic ESOCLCP to a stochastic mixed complementarity problem on nonegative orthant ($\StMCP$).
Enlightened by the idea from \cite{chen2011cvar}, we introduce the CVaR method to measure the loss of complementarity in the stochastic case. The merit function (\ref{que:SMIXCP}) is not required to equal zero almost surely for all $\omega \in \Omega$. Instead, a CVaR - based minimisation problem (\ref{que:mincvar}) is introduced to obtain a solution which is ``good enough" for the complementarity requirement of the original SMixCP. For solving the CVaR - based minimisation problem derived from the original SMixCP, smoothing function and sample average approximation methods are introduced and finally converted to the form in \eqref{que:cvarsaa}. Finally, a line search smoothing SAA algorithm is provided for finding the solution to this CVaR-based minimisation problem and it is illustrated by a numerical example.

Stochastic methods on complementarity problems were pioneered by Chen and Fukushima \cite{chen2005expected}, who introduced the idea of square norm of the
merit function which is still commonly used. However, this approach led to non-convexity and consequently increased the difficulty of solving ESOCLCP by algorithms. Our algorithm introduced in the previous section only guarantees finding a stationary point rather than a solution to the problem. We believe that the results in this paper will remain valid and further improved under strong convexity assumptions. This question will be a topic of future research.

\bibliographystyle{unsrt}{
\bibliography{stesoclcp}

\end{document}